\documentclass{article}
\usepackage{amssymb,amsmath,amsthm,graphicx}
\usepackage{amsfonts}
\usepackage[usenames,dvipsnames,svgnames,table]{xcolor}
\usepackage{color}
\newtheorem{theorem}{Theorem}

\newtheorem{proposition}[theorem]{Proposition}

\theoremstyle{definition}
\newtheorem{example}{Example}
\newtheorem{definition}{Definition}
\newtheorem{remark}{Remark}
\date{}

\title{\Large \textbf{Tribracket Modules}}

\author{Deanna Needell \footnote{Email: deanna@math.ucla.edu. Partially supported by NSF CAREER $\#1348721$.} \and
Sam Nelson\footnote{Email: sam.nelson@cmc.edu. Partially supported by Simons Foundation collaboration grant $\#316709$.}\and
Yingqi Shi\footnote{Email: yshi20@students.claremontmckenna.edu}}

\begin{document}
\maketitle

\begin{abstract} 
\textit{Niebrzydowski tribrackets} are ternary operations on sets satisfying
conditions obtained from the oriented Reidemeister moves such that the set
of tribracket colorings of an oriented knot or link diagram is an invariant
of oriented knots and links. We introduce \textit{tribracket modules} 
analogous to quandle/biquandle/rack modules and use these structures to 
enhance the tribracket counting invariant. We provide examples to illustrate 
the computation of the invariant and show that the enhancement is proper.
\end{abstract}

\parbox{5.5in} {\textsc{Keywords:} Niebrzydowski tribrackets, enhancements,
oriented knot and link invariants, tribracket modules

\smallskip

}

\section{Introduction}

In \cite{J} and \cite{M}, algebraic structures called \textit{quandles} (or 
\textit{distributive groupoids}) were introduced
as an abstraction of the Wirtinger presentation of the fundamental group
of the complement of a knot in $\mathbb{R}^3$. Colorings of knot diagrams
by elements of finite quandles define an integer-valued invariant known as
the \textit{quandle counting invariant}. Invariants of quandle-colored
knots, e.g. Boltzmann weights defined from quandle 2-cocycles, can be used 
to strengthen this invariant, defining new invariants known
as \textit{enhancements}. See \cite{EN} for more.

Initially defines in \cite{AG},
algebraic structures called \textit{quandle modules} were 
used in \cite{CEGS} to enhance the quandle counting invariant, inspiring later 
generalizations by one of the authors 
to the cases of rack modules in \cite{HHNYZ}, biquandle modules in \cite{BN}
and birack shadow modules in \cite{NP}, among others. In each of these 
cases, a counting invariant is enhanced with secondary colorings by elements
of a commutative ring with identity obeying an Alexander-style relation
which depends on the quandle colors at the crossing.

In \cite{MN}, the notion of using sets with ternary operations to define knot
invariants was considered, with colorings of regions in the planar complement
of a knot or link diagram by elements of structures known as \textit{ternary 
quasigroups}. These structures can be seen as an abstraction of the Dehn  
presentation of the knot group analogous to the way quandles abstract the 
Wirtinger presentation. In \cite{MN2}, ternary quasigroup invariants
were enhanced with a homology theory. A related structure called
\textit{biquasiles} was introduced in \cite{NN} by two of the authors 
with applications to surface-links in \cite{KN} by one of the authors. 
Recently ternary quasigroup operations known as \textit{Niebrzydowski 
tribrackets} have been studied with
additional generalizations to the cases of virtual knots in \cite{NP2} and
trivalent spatial graphs in \cite{GNT}.

In this paper we apply the idea behind quandle modules to the case of 
Niebrzydowski 
tribrackets, obtaining an infinite family of ehancements of the tribracket
counting invariant. The paper is organized as follows. In Section \ref{T}
we recall the basics of Niebrzydowski tribrackets and see some examples,
and introduce an enhancement for Alexander tribrackets. In Section
\ref{TM} we define tribracket modules and introduce the tribracket module
enhancement of the counting invariant. We compute some examples to show that
the enhancement is nontrivial. We conclude in Section \ref{Q} with some
questions for future work.

\section{Tribrackets}\label{T}


We begin with a defintion.

\begin{definition} (see e.g. \cite{NOO})
Let $X$ be a set. A \textit{horizontal tribracket} on $X$ is a 
map $[\ ,\ ,\ ]:X\times X\times X\to X$ satisfying
\begin{itemize}
\item[(i)] For any subset $\{a,b,c,d\}\subset X$, any three elements uniquely
determine the fourth such that $[a,b,c]=d$, and
\item[(ii)] 
\[[b,[a,b,c],[a,b,d]]=[c,[a,b,c],[a,c,d]]=[d,[a,b,d],[a,c,d]].\]
\end{itemize}
\end{definition}

\begin{example}
Let $X$ be any module over a commutative ring $R$ wth identity. Then any
pair of units $x,y\in R^{\times}$ defines a tribracket structure on $X$
by setting
\[[a,b,c]=-xya+xb+yc.\]
We call this an \textit{Alexander tribracket} and denote it by
$X=(R,x,y)$. Let us verify axiom (ii):
\begin{eqnarray*}
{}[b,[a,b,c],[a,b,d]] & = & -xyb+x(-xya+xb+yc)+y(-xya+xb+yd) \\
& = & (-x^2y-xy^2)a+x^2b+xyc+y^2d \\
{}[c,[a,b,c],[a,c,d]] & = & -xyc+x(-xya+xb+yc)+y(-xya+xc+yd) \\
& = & (-x^2y-xy^2)a+x^2b+xyc+y^2d \\
{}[d,[a,b,d],[a,c,d]] & = & -xyd+x(-xya+xb+yd)+y(-xya+xc+yd) \\
& = & (-x^2y-xy^2)a+x^2b+xyc+y^2d.
\end{eqnarray*}
\end{example}

\begin{example}
Let $G$ be a group. Then $G$ has the structure of a tribracket by setting
\[[a,b,c]=ba^{-1}c.\]
We call this a \textit{Dehn tribracket}. As with the Alexander case, let us
verify axiom (ii):
\begin{eqnarray*}
{}[b,[a,b,c],[a,b,d]] 
& = & [a,b,c]b^{-1}[a,b,d] \\
& = & ba^{-1}cb^{-1}ba^{-1}d \\
& = & ba^{-1}ca^{-1}d \\
{}[c,[a,b,c],[a,c,d]] 
& = & [a,b,c]c^{-1}[a,c,d] \\
& = & ba^{-1}cc^{-1}ca^{-1}d \\
& = & ba^{-1}ca^{-1}d \\
{}[d,[a,b,d],[a,c,d]] 
& = & [a,b,d]d^{-1}[a,c,d] \\
& = & ba^{-1}dd^{-1}ca^{-1}d \\
& = & ba^{-1}ca^{-1}d
\end{eqnarray*}
as required.
\end{example}

\begin{example}
We can specify a tribracket structure on a finite set $X=\{1,2,\dots, n\}$
with an operation 3-tensor, i.e. an ordered list of $n$ $n\times n$ matrices
with elements in $X$ such that the element in matrix $a$, row $b$, column $c$
is $[a,b,c]$. This notation enables us to compute with tribrackets for
which we lack algebraic formulas. For example, the set $X=\{1,2,3\}$
has tribracket structures including
\[\left[
\left[\begin{array}{rrr}
1 & 3 & 2 \\
2 & 1 & 3 \\
3 & 2 & 1
\end{array}\right],
\left[\begin{array}{rrr}
2 & 1 & 3 \\
3 & 2 & 1 \\
1 & 3 & 2
\end{array}\right],
\left[\begin{array}{rrr}
3 & 2 & 1 \\
1 & 3 & 2\\
2 & 1 & 3
\end{array}\right]
\right].\]
In this case for example, we verify axiom (ii) for the case $a=1,$ $b=2,$ $c=3,$ $d=1$ by the computation
\begin{eqnarray*}
{}[2, [1,2,3], [1,2,1]] & = &  [2,3,2] = 3, \\ 
{}[3, [1,2,3], [1,3,1]] & = &  [3,3,3] = 3 \quad \mathrm{and} \\ 
{}[1, [1,2,1], [1,3,1]] & = &  [1,2,3] = 3.
\end{eqnarray*}
\end{example}

The tribracket axioms are motivated by the Reidemeister moves using the 
following region coloring rule:

\[\includegraphics{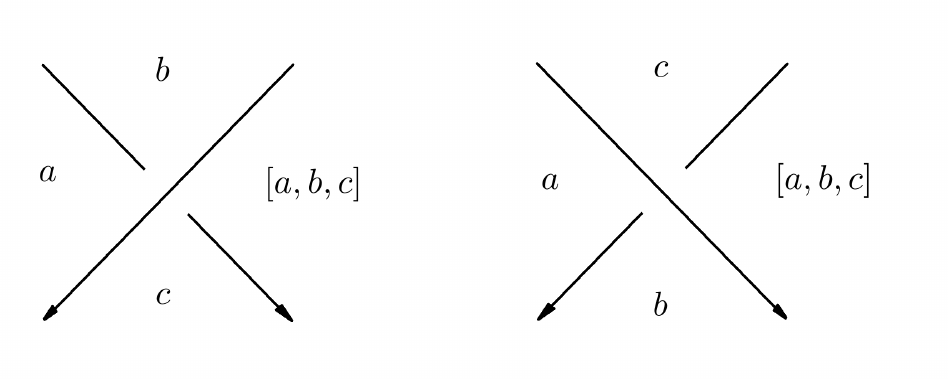}\]

We call the invertibility conditions in axiom (i) 
\textit{left, center} and \textit{right invertibility} for the ability to 
uniquely recover $a,b,$ and $c$ respectively in $[a,b,c]=d$ given the other 
three. These are the conditions required to guarantee that for every coloring
on one side of an oriented Reidemeister I or II move, there is a \textit{unique}
coloring of the diagram on the other side of the move which agrees with the
original coloring outside the neighborhood of the move. Axiom (ii) is
the condition required by the Reidemeister III move needed to complete a 
generating set of oriented Reidemesiter moves:
\[\includegraphics{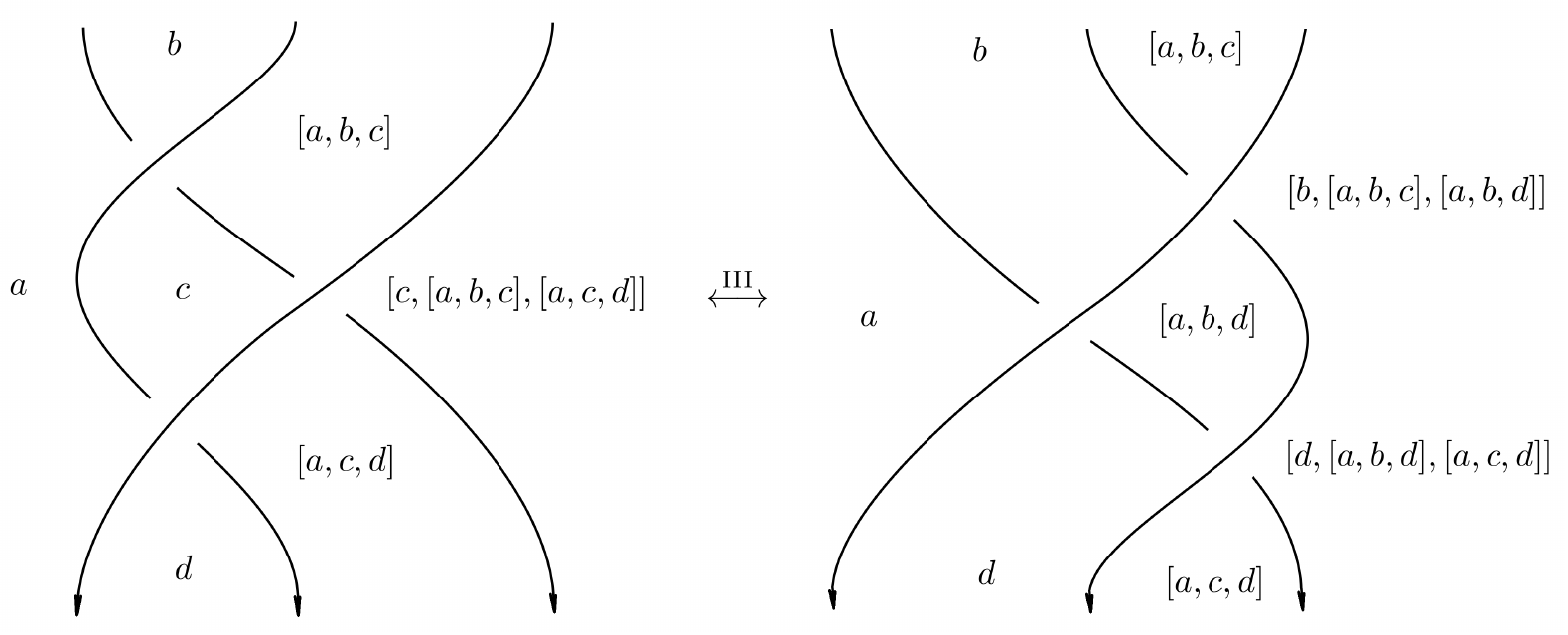}\]

It follows that for any tribracket $X$, the number of $X$-colorings of an 
oriented knot or link $L$ diagram is an integer-valued link invariant, which 
we call the \textit{tribracket counting invariant}, denoted 
$\Phi_X^{\mathbb{Z}}(L)$. We will denote the set of $X$-colorings of $L$ as
$\mathcal{C}_X(L)$, and we have $\Phi_X^{\mathbb{Z}}(L)=|\mathcal{C}_X(L)|$.

\begin{example}
If $X$ is an Alexander tribracket, we can compute $\Phi_X^{\mathbb{Z}}(L)$
using linear algebra. Let $X$ be the Alexander tribracket on $\mathbb{Z}_3$
with $x=1$ and $y=2$, so we have $[a,b,c]=a+b+2c$. Then the trefoil knot $3_1$
below has system of coloring equations
\[\raisebox{-0.5in}{\includegraphics{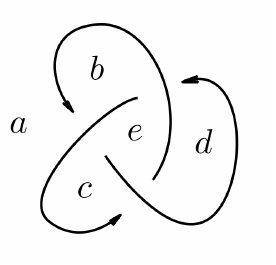}}
\quad\begin{array}{rcl}
{}[a,b,c] & = & e \\
{}[a,c,d] & = & e \\
{}[a,d,b] & = & e 
\end{array}
\]
and after row-reduction mod $3$, 
\[
\left[\begin{array}{rrrrr}
1 & 1 & 2 & 0 & 2 \\
1 & 0 & 1 & 2 & 2 \\
1 & 2 & 0 & 1 & 2
\end{array}\right]
\rightarrow
\left[\begin{array}{rrrrr}
1 & 0 & 0 & 2 & 2 \\
0 & 1 & 1 & 1 & 0 \\
0 & 0 & 0 & 0 & 0
\end{array}\right]
\]
we see that $\Phi_X^{\mathbb{Z}}(3_1)=3^3=27$. This distinguishes the trefoil
from the unknot $0_1$, which has $\Phi_X^{\mathbb{Z}}(0_1)=3^2=9$ $X$-colorings.
\end{example}

\begin{remark}
Writing the region Alexander tribracket coloring equations from an oriented 
link diagram as a system of linear equations yields a matrix from which
an Alexander matrix and the Alexander polynomial of the knot or link can be 
derived.
\end{remark}

An \textit{enhancement} of a counting invariant is a generally stronger
invariant from which we can recover the counting invariant. Any invariant
$\phi$ of $X$-colored knots and links defines an enhancement by taking
the multiset of $\phi$-values over the set of colorings $L_f$ 
of the knot or link $L$,
\[\Phi_X^{\phi,M}(L)=\{\phi(L_f)\ |\ L_f\in \mathcal{C}_X(L)\}.\]
Many such examples have been previously studied in the cases of quandle and
biquandle-colored knots and links, starting with the 2-cocycle enhancements
in \cite{CJKLS}; see \cite{EN} for a whole chapter of examples.

\begin{remark}
For Alexander tribrackets $X=(R,x,y)$, we can enhance the tribracket counting
invariant by setting $\phi(L_f)$ equal to the rank of the image submodule
of the coloring, analogous to the $(t,s)$-rack enhancements in \cite{CN}. 
The multiset of the ranks of these image submodules over the complete
set of colorings is the \textit{Alexander image enhancement} of the tribracket
counting invariant, 
\[\Phi_X^{A}(L)=\{\mathrm{rank}(\mathrm{Span}(\mathrm{Im}(f)))\ |\ f\in \mathcal{C}_X(L)\}.\]
\end{remark}

\section{Tribracket Modules}\label{TM}

We would like to ehance the tribracket counting invariant by finding an
invariant of $X$-colored oriented knot and link diagrams. To this end, we
make the following definition, analogous to the cases of quandles, racks,
and biracks in papers such as \cite{BN,CEGS,HHNYZ,NP}.

\begin{definition}\label{def:tbmod}
Let $X$ be a tribracket and $R$ a commutative ring with identity. A 
\textit{tribracket module structure} on $R$, also called an \textit{$X$-module},
is a choice of units $x_{a,b,c}$ and $y_{a,b,c}$ for each triple of elements
of $X$ satisfying the conditions
\begin{eqnarray*}
x_{c,[a,b,c],[a,c,d]}x_{a,b,c} & = & x_{d,[a,b,d],[a,c,d]}x_{a,b,d}\\
& = & x_{b,[a,b,c],[a,b,d]}x_{a,b,c} +y_{b,[a,b,c],[a,b,d]}x_{a,b,d} \\
& & -x_{b,[a,b,c],[a,b,d]}y_{b,[a,b,c],[a,b,d]} \\
y_{c,[a,b,c],[a,c,d]}y_{a,c,d} & = & y_{b,[a,b,c],[a,b,d]}y_{a,b,d}\\
& = & x_{d,[a,b,d],[a,c,d]}y_{a,b,d} +y_{d,[a,b,d],[a,c,d]}y_{a,c,d} \\
& & -x_{d,[a,b,d],[a,c,d]}y_{d,[a,b,d],[a,c,d]} \\
x_{b,[a,b,c],[a,b,d]}y_{a,b,c} & = & y_{d,[a,b,d],[a,c,d]}x_{a,c,d}\\
& = & x_{c,[a,b,c],[a,c,d]}y_{a,b,c} +y_{c,[a,b,c],[a,c,d]}x_{a,c,d} \\
& & -x_{c,[a,b,c],[a,c,d]}y_{c,[a,b,c],[a,c,d]}\\
      x_{c,[a,b,c],[a,c,d]}x_{a,b,c}y_{a,b,c} + y_{c,[a,b,c],[a,c,d]}x_{a,c,d}y_{a,c,d} 
& = & x_{b,[a,b,c],[a,b,d]}x_{a,b,c}y_{a,b,c} + y_{b,[a,b,c],[a,b,d]}x_{a,b,d}y_{a,b,d} \\
& = & x_{d,[a,b,d],[a,c,d]}x_{a,b,d}y_{a,b,d} + y_{d,[a,b,d],[a,c,d]}x_{a,c,d}y_{a,c,d} 
\end{eqnarray*}
for all $a,b,c,d\in X$. 
\end{definition}

A tribracket module over a tribracket $X=\{1,2,\dots, n\}$
is specified with a pair $V=(x,y)$ of $3$-tensors such that
the entries in matrix $a$, row $b$, column $c$ are $x_{a,b,c}$, $y_{a,b,c}$
respectively. 

\begin{remark}
The term ``module'' here follows the use of the term in \cite{AG} and other
previous work; it is justified by the fact that the invariant we define below
associates an $R$-module, generated by the regions in $L$ with relations
determined by the coloring, to each 
$X$-coloring $L_f$ of a link $L$.
\end{remark}

\begin{example}
A \textit{constant tribracket module} is one in which the $x_{a,b,c}$ and
$y_{a,b,c}$-values do not depend on $a,b,c\in X$. In this case, $V$ is
an Alexander tribracket on $R$ and the sticker colorings are independent 
of the $X$ colorings. For instance the tribracket 
\[X=\left[\left[\begin{array}{rr}
1 & 2 \\ 
2 & 1 
\end{array}\right],\left[\begin{array}{rr}
2 & 1 \\
1 & 2
\end{array}\right]\right]\]
has constant tribracket modules with $\mathbb{Z}_3$ coefficients including
\begin{eqnarray*}
V_1 & = & \left[\left[\begin{array}{rr}
1 & 1 \\ 
1 & 1 
\end{array}\right],\left[\begin{array}{rr}
1 & 1 \\
1 & 1
\end{array}\right]\right], \quad
\left[\left[\begin{array}{rr}
1 & 1 \\ 
1 & 1 
\end{array}\right],\left[\begin{array}{rr}
1 & 1 \\
1 & 1
\end{array}\right]\right],\\
V_2 & = & \left[\left[\begin{array}{rr}
1 & 1 \\ 
1 & 1 
\end{array}\right],\left[\begin{array}{rr}
1 & 1 \\
1 & 1
\end{array}\right]\right], \quad
\left[\left[\begin{array}{rr}
2 & 2 \\ 
2 & 2 
\end{array}\right],\left[\begin{array}{rr}
2 & 2 \\
2 & 2
\end{array}\right]\right],\\
V_3 & = & \left[\left[\begin{array}{rr}
2 & 2 \\ 
2 & 2 
\end{array}\right],\left[\begin{array}{rr}
2 & 2 \\
2 & 2
\end{array}\right]\right], \quad
\left[\left[\begin{array}{rr}
1 & 1 \\ 
1 & 1 
\end{array}\right],\left[\begin{array}{rr}
1 & 1 \\
1 & 1
\end{array}\right]\right] \quad \mathrm{and}\\
V_4 & = & \left[\left[\begin{array}{rr}
2 & 2 \\ 
2 & 2 
\end{array}\right],\left[\begin{array}{rr}
2 & 2 \\
2 & 2
\end{array}\right]\right], \quad
\left[\left[\begin{array}{rr}
2 & 2 \\ 
2 & 2 
\end{array}\right],\left[\begin{array}{rr}
2 & 2 \\
2 & 2
\end{array}\right]\right].\end{eqnarray*}
\end{example}

The tribracket module axioms are motivated by the Reidemeister moves with
the coloring scheme below. Given an oriented knot or link diagram with
a tribracket coloring, we put a secondary labeling on the regions with a 
\textit{sticker} (an element of $R$, represented in our diagrams as
an element of $R$ surrounded by a box) in each region. 
\[\includegraphics{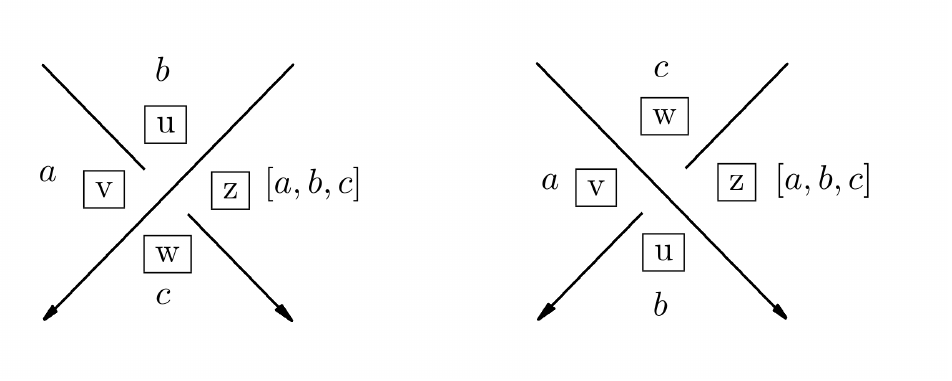}\]
The sticker colorings must then satisfy the rule
\[z=-x_{a,b,c}y_{a,b,c}v+x_{a,b,c}u+y_{a,b,c}w,\]
a customized Alexander tribracket-style coloring with coefficients
depending on the $X$-colors at the crossing.

\begin{proposition}
Let $X$ be a tribracket, $R$ a commutative ring with identity and $V$ an
$X$-module over $R$. Then sticker colorings of the regions in an 
oriented knot diagram's planar complement are in one-to-one correspondence
before and after $X$-colored Reidemeister moves.
\end{proposition}

\begin{proof} Recall (see \cite{P} for example) that the set of all four 
oriented Reidemeister I moves, all four oriented Reidemeister II moves, and 
the all-positive Reidemeister III move forms a generating set of oriented 
Reidemeister moves.
Invertibility of $x_{a,b,c}$ and $y_{a,b,c}$ satisfies the claim for Reidemeister
I and II moves; let us illustrate with one of the four oriented 
Reidemeister II moves.
\[\includegraphics{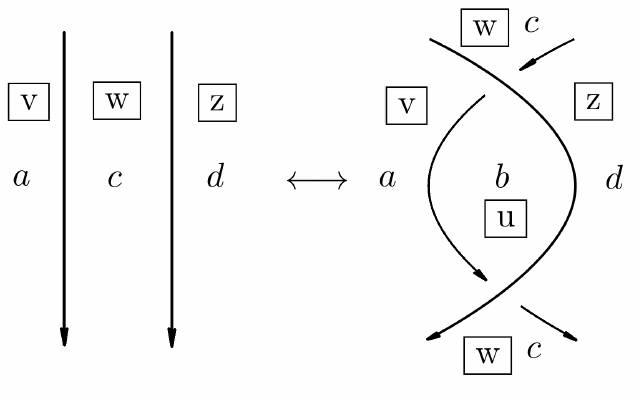}\]
The condition we need for uniqueness of the sticker $u$ given $v,w$ and $z$
is that the equation
\[z=-x_{a,b,c}y_{a,b,c}v+x_{a,b,c}u+y_{a,b,c}w\]
should be solvable for $u$ in terms of $v,w$ and $z$; this is possible
provided $x_{a,b,c}$ is invertible in $R$:
\[u=y_{a,b,c}v+x_{a,b,c}^{-1}z-x_{a,b,c}^{-1}y_{a,b,c}w\]
The other Reidemeister I and II moves are similar. It remains only to 
verify for the all-positive Reidemeister III move. 
\[\includegraphics{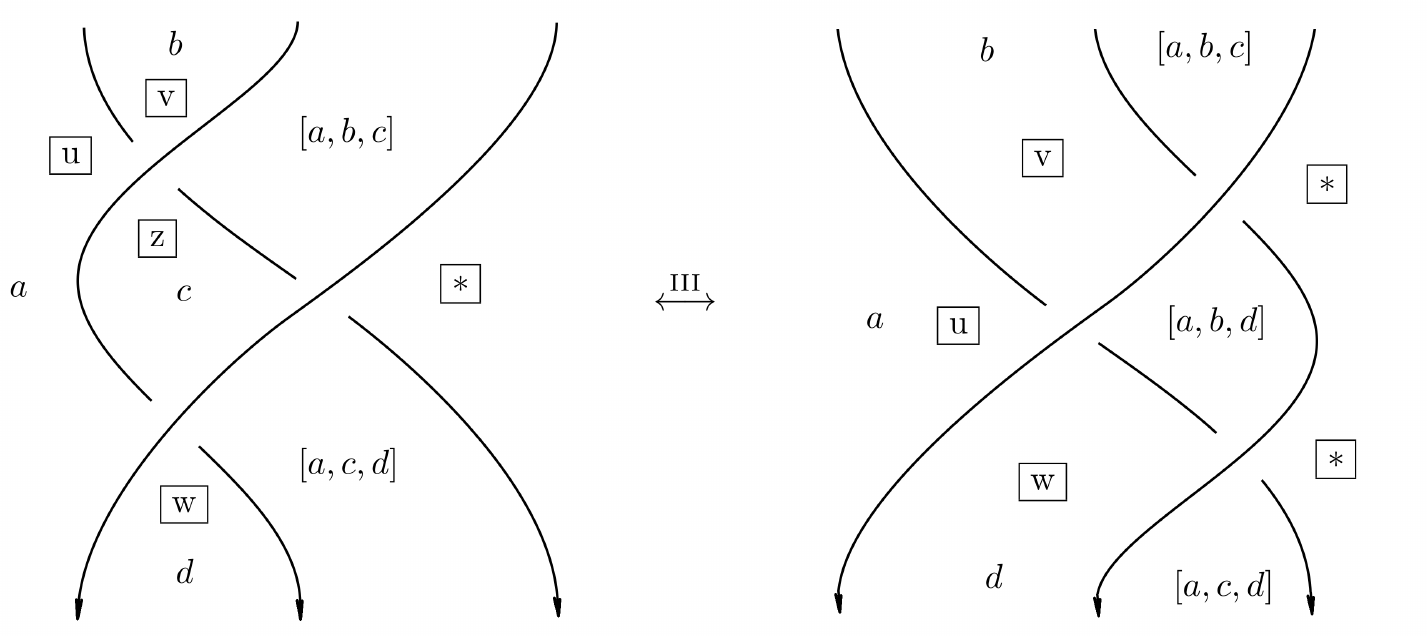}\]
The region marked $\framebox{$\ast$}$ gets three sticker colorings 
-- one on the left side of the move and two on the right --
which must all agree. Each of 
these is an expression in the independent variables $u,v,z,w$, so we can
compare these coefficients to obtain the necessary equations, i.e., the 
conditions in Defintion \ref{def:tbmod}.
\end{proof}

\begin{definition}
Let $L$ be an oriented link diagram, $X$ a tribracket, $R$ a commutative ring 
with identity and $V=(x,y)$ an $X$-module structure on $R$. For each 
$X$-coloring $f\in \mathcal{C}_X(L)$ of $L$, let $A_f$ be the coefficient 
matrix of the homogeneous system of linear equations
in the $R$-module generated by the regions of $L$ determined by the
sticker coloring equations. Then the 
\textit{tribracket module multiset enhancement} of the tribracket counting 
invariant is the multiset
\[\Phi_X^{V,M}(L)=\{|\mathrm{Ker} \ A_f|\ :\ f\in\mathcal{C}_X(L)\}\]
or if $R$ is infinite,
\[\Phi_X^{V,M}(L)=\{\mathrm{rank}(\mathrm{Ker} \ A_f)\ :\ f\in\mathcal{C}_X(L)\}.\]
We can optionally convert these multisets to ``polynomial''
form by replacing multiplicities with coefficients and elements with exponents
of a formal variable $u$
\[\Phi_X^{V}(L)=\sum_{f\in\mathcal{C}_X(L)}u^{|\mathrm{Ker} \ A_f|}\]
or if $R$ is infinite,
\[\Phi_X^{V}(L)=\sum_{f\in\mathcal{C}_X(L)}u^{\mathrm{rank}(\mathrm{Ker} \ A_f)}.\]
This notation has the advantage that evaluation at $u=1$ yields the original 
counting invariant and provides easier visual comparison of invariant values.

\end{definition}

By construction, we have the following proposition:
\begin{proposition}
For any $X$-module $V$ over a tribracket $X$ and commutative ring with 
identity $R$, $\Phi_X^{V,M}(L)$ and $\Phi_X^{V}(L)$ are invariants of oriented
knots and links.
\end{proposition}

\begin{example}
If $V$ is a constant $X$-module, then $|\mathrm{Ker}\ A_f|$ is just
the number of colorings of $L$ by the Alexander tribracket $A$ on $R$ with
parameters $(x,y)$ and we have
\[\Phi_X^{V}(L)=|\Phi_X^{\mathbb{Z}}(L)| u^{|\Phi_A^{\mathbb{Z}}(L)|}.\]
\end{example}

\begin{example}\label{ex:nt}
Let $V$ be an $X$-module with coefficients in a finite ring $R$. The unlink
of $n$ components has $n+1$ regions with no crossings and hence no
restrictions on $X$-colorings, so there are $|X|^{n+1}$ region colorings. 
Each of these has similarly no restrictions on sticker colorings, so there are
$|R|^{n+1}$ sticker colorings for each region coloring. Hence,
the value of $\Phi_X^{V}(L)$ on the unlink of $n$ components is
$\Phi_X^{V}(L)=|X|^{n+1}u^{|R|^{n+1}}$.
\end{example}

\begin{example}\label{ex:nt8}
Let $X$ be the set $\{1,2\}$ with tribracket operation given by
\[\left[\left[\begin{array}{rr}
1 & 2 \\ 
2 & 1 
\end{array}\right],\left[\begin{array}{rr}
2 & 1 \\
1 & 2
\end{array}\right]\right].\]
The trefoil knot $3_1$ has four $X$-colorings:
\[\includegraphics{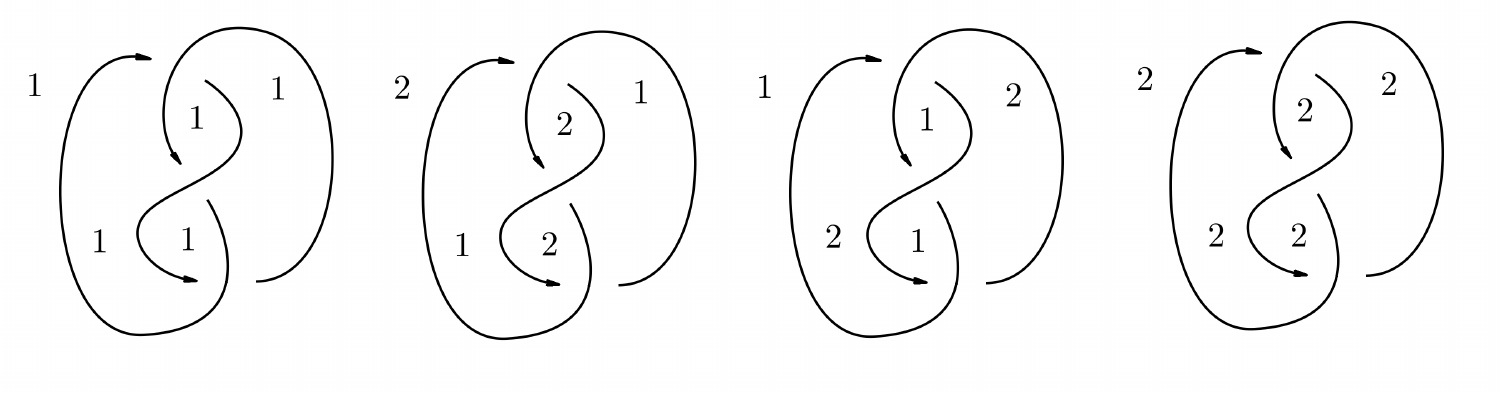}\] 
Then we compute via \texttt{python} that $X$ has modules with $\mathbb{Z}_3$
coefficients including
\[V=
\left[\left[\begin{array}{rr}
2 & 2 \\ 
2 & 1 
\end{array}\right],\left[\begin{array}{rr}
1 & 2 \\
2 & 2
\end{array}\right]\right], \quad
\left[\left[\begin{array}{rr}
1 & 2 \\ 
2 & 2 
\end{array}\right],\left[\begin{array}{rr}
2 & 2 \\
2 & 1
\end{array}\right]\right].
\]
Consider the $X$-coloring of the trefoil with all regions colored $1\in X$:
\[\includegraphics{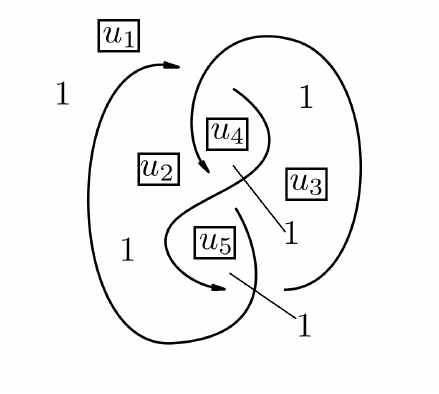}\] 
Let us compute the set of sticker colorings. We obtain linear system of 
coloring equations over $R=\mathbb{Z}_3$
\[\begin{array}{rcl}
-x_{11}y_{11}u_2+x_{11}u_1+y_{11}u_4 & = & u_3 \\
-x_{11}y_{11}u_2+x_{11}u_1+y_{11}u_4 & = & u_3 \\
-x_{11}y_{11}u_2+x_{11}u_1+y_{11}u_4 & = & u_3 \\
\end{array}\Rightarrow 
\begin{array}{rcl}
-(1)(2)u_2+1u_1+2u_4 & = & u_3 \\
-(1)(2)u_2+1u_4+2u_5 & = & u_3 \\
-(1)(2)u_2+1u_5+2u_1 & = & u_3 \\
\end{array}\]
which via row-reduction over $\mathbb{Z}_3$ 
\[
\Rightarrow 
\left[\begin{array}{rrrrr}
1 & 1 & 2 & 2 & 0 \\
0 & 1 & 2 & 1 & 2 \\
2 & 1 & 2 & 0 & 1 \\
\end{array}\right]
\Rightarrow 
\left[\begin{array}{rrrrr}
1 & 0 & 0 & 1 & 1 \\
0 & 1 & 2 & 1 & 2 \\
0 & 0 & 0 & 0 & 0 \\
\end{array}\right]
\]
has kernel of dimension $3$. Similarly, the other $X$-colorings have 
$|R|^3=3^3=27$ colorings.
In particular the trefoil has $\Phi_X^{V}(3_1)=4u^{27}$ which is different from 
the unknot's value of $4u^9$, and the enhancement detects the difference between
the unknot and the trefoil.
\end{example}

\begin{example}
Let $X$ be the tribracket in example \ref{ex:nt8}. Via \texttt{python} 
computation we selected tribracket modules $V_1$ and $V_2$
with coffecients in 
$\mathbb{Z}_3$ and $V_3$ with coefficients in $\mathbb{Z}_8$, 
\begin{eqnarray*}V_1&=&
\left[\left[\begin{array}{rr}
2 & 1 \\ 
2 & 2 
\end{array}\right],\left[\begin{array}{rr}
2 & 2 \\
1 & 2
\end{array}\right]\right], \quad
\left[\left[\begin{array}{rr}
2 & 2 \\ 
1 & 1 
\end{array}\right],\left[\begin{array}{rr}
1 & 1 \\
2 & 2
\end{array}\right]\right], \\
V_2 & = & 
\left[\left[\begin{array}{rr}
1 & 1 \\ 
1 & 1 
\end{array}\right],\left[\begin{array}{rr}
1 & 1 \\
1 & 1
\end{array}\right]\right], \quad
\left[\left[\begin{array}{rr}
1 & 1 \\ 
2 & 2 
\end{array}\right],\left[\begin{array}{rr}
2 & 2 \\
1 & 1
\end{array}\right]\right]
\quad \mathrm{and}\\
V_3 & = & 
\left[\left[\begin{array}{rr}
1 & 3 \\ 
1 & 7 
\end{array}\right],\left[\begin{array}{rr}
7 & 1 \\
3 & 1
\end{array}\right]\right], \quad
\left[\left[\begin{array}{rr}
1 & 5 \\ 
1 & 1 
\end{array}\right],\left[\begin{array}{rr}
1 & 1 \\
5 & 1
\end{array}\right]\right],
\end{eqnarray*}
and computed the $\Phi_X^V(L)$ values for the prime links of up to seven 
crossings at the knot atlas \cite{KA}. The results are collected in the tables.
\[
\begin{array}{r|l}
\Phi_X^{V_1}(L) & L \\ \hline\hline
2u^9+6u^{27} & L2a1, L4a1, L5a1, L6a2, L7a4, L7a6 \\
2u^9+4u^{27}+2u^{81} & L7a2, L7a3, L7n1, L7n2 \\
8u^{27} & L6a1, L6a3, L7a1, L7a5 \\ \hline
8u^{27}+8u^{81} & L6a5 \\
2u^9+6u^{27}+8u^{81} & L6n1, L7a7 \\
2u^9+14u^{81} & L6a4 \\
\end{array}
\]
\[
\begin{array}{r|l}
\Phi_X^{V_2}(L) & L \\ \hline\hline
6u^9+2u^{27} & L2a1, L62, L7a6 \\
2u^9+6u^{27} & L4a1, L5a1, L7a2, L7a3, L7a4, L7n1, L7n2 \\
4u^9+4u^{27} & L6a3, L7a5 \\
8u^{27} & L6a1, L7a1 \\ \hline
2u^9+6u^{27}+8u^{81} & L6a4 \\
6u^9+8u^{27}+2u^{81} & L6a5 \\
8u^9+6u^{27}+2u^{81} & L6n1, L7a7 \\
\end{array}
\]
\[
\begin{array}{r|l}
\Phi_X^{V_3}(L) & L \\ \hline\hline
2u^{128}+4u^{256}+2u^{512} & L2a1, L6a2, L6a3, L7a5, L7a6 \\
2u^{256}+6u^{512} & L4a1, L6a1, L7a2,L7n1 \\
8u^{512} & L5a1, L7a1, L7a3, L7a4, L7n2  \\\hline
2u^{256}+6u^{1024}+6u^{2048}+2u^{4096} & L6a5, L6n1,L7a7 \\
2u^{1024}+6u^{2048}+8u^{4096} & L6a4
\end{array}
\]
\end{example}

\begin{example}
For our final example let $X$ be the 4-element tribracket
\[
\left[
\left[\begin{array}{rrrr}
4 & 3 & 2 & 1 \\
2 & 4 & 1 & 3 \\
3 & 1 & 4 & 2 \\
1 & 2 & 3 & 4 
\end{array}\right],
\left[\begin{array}{rrrr}
3 & 1 & 4 & 2 \\
4 & 3 & 2 & 1 \\
1 & 2 & 3 & 4 \\
2 & 4 & 1 & 3 \\
\end{array}\right],
\left[\begin{array}{rrrr}
2 & 4 & 1 & 3 \\
1 & 2 & 3 & 4 \\
4 & 3 & 2 & 1 \\
3 & 1 & 4 & 2 \\
\end{array}\right],
\left[\begin{array}{rrrr}
1 & 2 & 3 & 4 \\
3 & 1 & 4 & 2 \\
2 & 4 & 1 & 3 \\
4 & 3 & 2 & 1
\end{array}\right]\right]
\]
with the module $V$ with $\mathbb{Z}_3$ coefficients specified by
\[
\left[
\left[\begin{array}{rrrr}
1 & 1 & 1 & 1 \\
2 & 1 & 1 & 2 \\
2 & 1 & 1 & 2 \\
1 & 1 & 1 & 1
\end{array}\right],
\left[\begin{array}{rrrr}
1 & 2 & 2 & 1 \\
1 & 1 & 1 & 1 \\
1 & 1 & 1 & 1 \\
1 & 2 & 2 & 1
\end{array}\right],
\left[\begin{array}{rrrr}
1 & 2 & 2 & 1 \\
1 & 1 & 1 & 1 \\
1 & 1 & 1 & 1 \\
1 & 2 & 2 & 1
\end{array}\right],
\left[\begin{array}{rrrr}
1 & 1 & 1 & 1 \\
2 & 1 & 1 & 2 \\
2 & 1 & 1 & 2 \\
1 & 1 & 1 & 1 
\end{array}\right]
\right],
\]\[
\left[
\left[\begin{array}{rrrr}
2 & 2 & 2 & 2 \\
1 & 1 & 1 & 1 \\
1 & 1 & 1 & 1 \\
2 & 2 & 2 & 2 \\
\end{array}\right],
\left[\begin{array}{rrrr}
1 & 1 & 1 & 1 \\
2 & 2 & 2 & 2 \\
2 & 2 & 2 & 2 \\
1 & 1 & 1 & 1
\end{array}\right],
\left[\begin{array}{rrrr}
1 & 1 & 1 & 1 \\
2 & 2 & 2 & 2 \\
2 & 2 & 2 & 2 \\
1 & 1 & 1 & 1
\end{array}\right],
\left[\begin{array}{rrrr}
2 & 2 & 2 & 2 \\
1 & 1 & 1 & 1 \\
1 & 1 & 1 & 1 \\
2 & 2 & 2 & 2
\end{array}\right]
\right].
\]
We computed the invariant on prime knots with up to eight crossings and links
with up to seven crossings. The results are collected in the table. In
particular the knots in the table all have counting invariant value 16 but
are sorted into three classes by the enhancement, while the invariant is quite 
effective at distinguishng the links in the table.
\[\begin{array}{r|l}
\Phi_X^V(L) & L \\ \hline
16u^9 & 4_1, 5_1, 5_2, 6_2, 6_3, 7_1, 7_2, 7_3, 7_5, 7_6, 8_1, 8_2, 8_3, 8_4, \\
 & 8_6, 8_7, 8_8, 8_9, 8_{12}, 8_{13}, 8_{14}, 8_{16}, 8_{17} \\
8u^9+8u^{27} & 3_1, 6_1, 7_4,7_7, 8_5, 8_{10}, 8_{11}, 8_{15}, 8_{19}, 8_{20}, 8_{21} \\
8u^9+8u^{81} & 8_{18} \\\hline
16u^9+16u^{27} & L2a1, L6a2, L7a6 \\
32u^{27} & L6a3, L7a5\\\hline
16u^9+32u^{27}+16u^{81} & L7a3, L7n1, L7n2 \\
16u^9+48u^{27} & L4a1, L5a1, L7a4 \\
32u^9+32u^{81} & L6n1, L7a7 \\
32u^{27}+32u^{81} & L6a5 \\\hline
64u^{27} & L6a1, L7a1  \\\hline
32u^9+ 224u^{81} & L6a4 \\
\end{array}
\]
\end{example}

\section{Questions}\label{Q}

We conclude with some questions and directions for future work.

First and foremost, more efficient methods than our (relatively) brute-force
axiom testing for finding tribracket modules would be highly desirable. While
even the relatively small examples we have found are fairly good at 
distinguishing classical knots and links, we expect that modules over larger
finite rings or infinite rings should yield even stronger invariants.

What is the relationship between tribracket modules and tribracket cocycles?
In the case of racks, rack modules are closely related to structures known
as \textit{dynamical cocycles} \cite{AG}; what is the appropriate definition
for tribracket dynamical cocycles?

As in the case of \cite{CN2}, we can consider tribracket modules with 
coefficients in a polynomial algebra as defining a kind of tribracket-colored
Alexander polynomial for each coloring. Do these invariants satisfy skein
relations? Do their coefficients define Vassiliev invariants?

\bibliography{dn-sn-ys}{}
\bibliographystyle{abbrv}

\bigskip

\noindent
\textsc{Department of Mathematics \\
Univ. of California, Los Angeles \\
520 Portola Plaza, Los Angeles, CA 90095 }

\medskip

\noindent
\textsc{Department of Mathematical Sciences \\
Claremont McKenna College \\
850 Columbia Ave. \\
Claremont, CA 91711}

\end{document}